\documentclass[11pt, oneside]{article}
\usepackage{amsfonts,amssymb,amsmath,latexsym,amsthm}
\usepackage{mathrsfs}
\usepackage{color}
\usepackage[colorlinks]{hyperref}
\usepackage{enumerate,indentfirst,mathtools,verbatim,authblk}
\usepackage{tikz}
\usepackage{psfrag}
\usepackage{graphicx}
\usepackage{bm}
\usepackage{appendix}
\usepackage[rflt]{floatflt}
\usepackage{float}
\usepackage[square, comma, sort&compress, numbers]{natbib}
\numberwithin{equation}{section}
\allowdisplaybreaks

\newtheorem{theorem}{Theorem}[section]

\newtheorem{lemma}[theorem]{Lemma}

\theoremstyle{definition}
\newtheorem{definition}[theorem]{Definition}
\theoremstyle{assumption}

\theoremstyle{remark}
\newtheorem{remark}[theorem]{Remark}
\numberwithin{equation}{section}

\usepackage{amssymb,amsmath}
\begin{document}
\title{\LARGE\bf {Decay rates for a variable-coefficient wave equation with nonlinear time-dependent damping}}
\author{Menglan Liao\thanks{Corresponding author: Menglan Liao \newline \hspace*{6mm}{\it Email
addresses:} liaoml@hhu.edu.cn}}
\affil{College of Science, Hohai University, Nanjing, {\rm210098}, China}

\renewcommand*{\Affilfont}{\small\it}
\date{} \maketitle
\vspace{-20pt}

{\bf Abstract:}  In this paper, a class of variable-coefficient wave equations equipped with time-dependent damping and the nonlinear source is considered. We show that the total energy of the system 
decays to zero with an explicit and precise decay rate estimate under different assumptions on the feedback with the help of the method of weighted energy integral. 

 

\thispagestyle{empty}
\section{Introduction}
In this paper, we are concerned with the following initial-boundary value problem for a class of variable-coefficient wave equations with time-dependent damping 
\begin{equation}
\label{1.1}
\begin{cases}
      u_{tt}-\mathcal{A}u+\gamma(t)g(u_t)=f(u)& \quad x\in \Omega,~t>0,  \\
      u(x,t)=0&\quad x\in \partial\Omega,~t>0,\\
     u(x,0)=u_0(x),~u_t(x,0)=u_1(x)&\quad x\in \Omega,
\end{cases}
\end{equation}
where $\Omega\subset \mathbb{R}^3$, for simplicity, is a bounded domain with $C^\infty$ boundary.  In particular, present results can be extended to bounded domain in $\mathbb{R}^n$ by adopting the Sobolev embeddings, and by adjusting some parameters imposed. Here $\gamma(t)g(u_t)$ is a time-dependent damping term, and $f\in C^1(\mathbb{R})$ such that $|f'(s)|\le C(|s|^{p-1}+1)$ with $1\le p<6$. To simplify computations, we choose $f(u):=|u|^{p-1}u$ with $1\le p<6$ in this paper. The second-order differential operator $\mathcal{A}$ is defined by
\[\mathcal{A}u=\mathrm{div}(A(x)\nabla u)=\sum_{i,j=1}^\infty\frac{\partial}{\partial x_i}\Big(a_{ij}(x)\frac{\partial u}{\partial x_j}\Big),\]
where $A(x)=(a_{i,j}(x))$ is symmetric and positive definite matrices with $a_{i,j}(x)\in C^\infty(\bar{\Omega})$ and  satisfies the uniform ellipticity conditions
\begin{equation}
\label{1.2}
\sum_{i,j=1}^na_{i,j}(x)\xi_i\xi_j\ge \omega \sum_{i=1}^n\xi_i^2,\quad x\in \bar{\Omega},~\omega>0.
\end{equation}
It is easily verified that the bilinear form $a(\cdot,\cdot): H_0^1(\Omega)\times H_0^1(\Omega)\to \mathbb{R}$ defined by
\[a(u,v)=\sum_{i,j=1}^n\int_\Omega a_{i,j}(x)\frac{\partial u}{\partial x_j}\frac{\partial u}{\partial x_i}dx=\int_\Omega A\nabla u\cdot \nabla vdx\]
is symmetric and continuous. Further, it follows from \eqref{1.2} that
\begin{equation}
\label{1.3}
a(u,v)\ge \omega\|\nabla u\|_2^2.
\end{equation}

 The problem of asymptotic stability of solutions of dissipative 
wave systems equipped with time-dependent nonlinear damping forces has been given a lot of attention.  P. Pucci and J. Serrin \cite{PS1996} investigated  the following nonlinear damped wave system with Dirichlet data
\begin{equation}
\label{421}
u_{tt}-\Delta u+Q(t,x,u,u_t)+f(x,u)=0,
\end{equation}
where the function $Q$ represents a \textit{nonlinear damping} satisfying $(Q(t,x,u,v),v)\ge 0$, $f$ is a \textit{restoring force}(that is $(f(x,u),u)\ge 0$). Based on the a priori existence of a suitable auxiliary function, they  proved the natural energy $Eu(t)$ associated with solutions of the system  satisfies $\lim_{t\to +\infty}Eu(t)=0$. Problem \eqref{421} is well-known in a variety of models in mathematical physics, for instance, elastic vibrations in a dissipative medium, the telegraphic equation, and the damped Klein-Gordon equation.    P. Pucci and J. Serrin \cite{PS1998} studied again problem \eqref{421} not only for potential energies which arise from restoring forces, but also for the effect of \emph{amplifying forces}. They pointed out that global asymptotic stability can no longer be expected, and should be replaced by local stability.  However, whether an explicit and precise decay rate estimate of the total energy of the system  can be obtained was unknown.  P. Martinez \cite{M2000} considered the following time-dependent dissipative systems
\[u_{tt}-\Delta u+\rho(t,u_t)=0\]
with Dirichlet boundary, where $\rho:\mathbb{R}^+\times \mathbb{R}\to \mathbb{R}$ is  a continuous function differentiable on $\mathbb{R}^+\times(-\infty,0)$.  By generalizing the method introduced 
to study autonomous wave equation damped with a boundary nonlinear velocity feedback $\rho(u_t)$ in \cite{M1999}, P. Martinez obtained that the total energy of the system decays to zero with an explicit and precise decay rate estimate under sharp assumptions on the feedback. M. Daoulatli \cite{D2011} studied problem \eqref{1.1} without $f(u)$, and they illustrated that if given suitable conditions on the nonlinear terms, and the damping is modeled by a continuous monotone function without any growth restrictions imposed at the origin and infinity, the decay rate of the energy functional was obtained by solving a nonlinear non-autonomous ODE.  The Cauchy problem for second order hyperbolic evolution equations with restoring force in a Hilbert space, under the effect of nonlinear time-dependent damping has also been studied, the interested readers can  refer to papers \cite{LX2020,LX2021,2LX2021,LX2023}.

It is certainly beyond the scope of the present paper to give a comprehensive review for dissipative wave systems equipped with time-dependent nonlinear damping forces. However, the literature on energy decay estimates is rare for the weak solution of the variable coefficient wave equations when it is concerned with the interaction between time-dependent damping and source term. Inspired by the paper \cite{M2000}, in this paper, the primary goal is to establish the energy decay rates when the time-dependent damping satisfies different assumptions.   Because of the  the interaction between time-dependent damping and source term, some new difficulties need to be solved. The outline of the paper is as follows. In Section \ref{sec2}, we shall give some assumptions, main results and several remarks. In Section \ref{sec3}, we prove that the local (in time) existence of the weak solution can be extended globally. Sections \ref{sec4} and \ref{sec5} are used to prove the energy decay rates.

\section{Preliminaries and main results}\label{sec2}

Throughout the paper, denote by $M$ the optimal embedding constant for the embedding theorem $H_0^1(\Omega) \hookrightarrow L^{p+1}(\Omega).$  The symbol $C$ is a generic positive constant, which may be different in various positions. $C_i$ $(i=0,1,2,\cdots,33)$ represent some positive constants.  We impose the following assumptions:

$\mathbf{(H_1)}$ $\gamma\in W_{loc}^{1,\infty}(\mathbb{R}^+)$ is bounded and nonnegative on $\mathbb{R}^+=[0,+\infty).$

$\mathbf{(H_2)}$  $g$ is continuous and monotone increasing feedback with $g(0)=0$. In addition, there exist positive constants $b_1$ and $b_2$ such that 
\[b_1|s|^{m+1}\le g(s)s\le b_2|s|^{m+1},\quad\text{ where }m\ge 1 \text{ and }|s|>1.\]


$\mathbf{(H_3)}$ $p\frac{m+1}{m}<6.$

$\mathbf{(H_4)}$ $u_0(x)\in H_0^1(\Omega),~u_1(x)\in L^2(\Omega).$

\begin{definition}[Weak solution]
A function $u(t):=u(t,x)$ is said to a weak solution of problem \eqref{1.1} on $[0,T]$ if  $u\in C([0,T];H_0^1(\Omega))$ with $u_t\in C([0,T];L^2(\Omega))\cap L^{m+1}(\Omega\times(0,T))$. In addition, for all $t\in [0,T]$, 
\begin{equation}
\label{2.1}
\begin{split}
&(u_t(t),\phi(t))-(u_t(0),\phi(0))+\int_0^ta(u(s),\phi(s))ds-\int_0^t\int_\Omega u_t(s)\phi_t(s)dxds\\
&\quad+\int_0^t\int_\Omega \gamma(t)g(u_t(s))\phi_t(s)dxds=\int_0^t\int_\Omega f(u(s))\phi(s)dxds
\end{split}
\end{equation}
for $\phi\in \{\phi:\phi\in C([0,T];H_0^1(\Omega))\cap L^{m+1}(\Omega\times(0,T))\text{ with }\phi_t\in C([0,T];L^2(\Omega))\}.$

\end{definition}

\begin{theorem}[Local existence]\label{thle}
Let $\mathbf{(H_1)}-\mathbf{(H_4)}$ hold, then there exists a local (in time) weak solutions $u(t)$ to problem \eqref{1.1} for $T>0$ depending on the initial quadratic energy $\mathscr{E}(0)$. Moreover, the following identity holds
\begin{equation}
\label{2.2}
\mathscr{E}(t)+\int_0^t\int_\Omega \gamma(t)g(u_t)u_tdxds=\mathscr{E}(0)+\int_0^t\int_\Omega f(u(s))u_t(s)dxds
\end{equation}
with the quadratic energy is defined by
\begin{equation}
\label{2.3}
\mathscr{E}(t)=\frac12\|u_t(t)\|_2^2+\frac12a(u(t),u(t)).
\end{equation}
\end{theorem}
The following result illustrates that when the damping dominates the source, then the solution is global in time. 
\begin{theorem}[Global existence \uppercase\expandafter{\romannumeral1}]\label{thge1}
In addition to $\mathbf{(H_1)}-\mathbf{(H_4)}$, if $u_0(x)\in L^{p+1}(\Omega)$ and $m\ge p$, then the weak solution of problem \eqref{1.1} is global in time.
\end{theorem}
Theorem \ref{thle} can be proved directly by employing the theory of monotone operators and nonlinear semigroups combined with energy methods. Theorem \ref{thge1} also follows from a standard continuation argument from ODE theory.  The interested readers can follow from line to line as shown in  \cite{GR2014}(see also a recent paper \cite{H2021}) with slight differences to achieve it. The proof of Theorem \ref{thle} and Theorem \ref{thge1} is not the point in question, so we omit it.

For $u\in L^{p+1}(\Omega)$, define the total energy $E(t)$ by
\begin{equation}
\label{2.4}
E(t):=\mathscr{E}(t)-\frac{1}{p+1}\|u\|^{p+1}=\frac12\|u_t(t)\|_2^2+\frac12a(u(t),u(t))-\frac{1}{p+1}\|u\|_{p+1}^{p+1}.
\end{equation}
Then \eqref{2.2} can be transferred as 
\begin{equation}
\label{411}
E(t)+\int_0^t\int_\Omega \gamma(t)g(u_t)u_tdxds=E(0).
\end{equation}
Further, we obtain the total energy $E(t)$ is monotone decreasing in time, and 
\begin{equation}
\label{4111}
E'(t)=-\int_\Omega \gamma(t)g(u_t)u_tdx.
\end{equation}
\begin{theorem}[Global existence \uppercase\expandafter{\romannumeral2}]\label{thge2}
Let $1<p\le 5$ and $\mathbf{(H_1)}-\mathbf{(H_4)}$ hold. Assume \[0<E(0)<\Big(\frac12-\frac{1}{p+1}\Big)\Big(\frac{\omega^{\frac{p+1}{2}}}{M^{p+1}}\Big)^{\frac{2}{p-1}},\quad a(u_0,u_0)<\Big(\frac{\omega^{\frac{p+1}{2}}}{M^{p+1}}\Big)^{\frac{2}{p-1}},\] then the weak solution of problem \eqref{1.1} is global.
\end{theorem}
We need to impose additional conditions to discuss the energy decay rates.

$\mathbf{(H_5)}$ $\gamma:\mathbb{R}^+\to \mathbb{R}^+$ a nonincreasing function of class $C^1$
and $\int_0^\infty \gamma(t)dt=\infty.$

$\mathbf{(H_6)}$  There exists a strictly increasing and odd function $h\in C^1[1,1]$  such that 
\[|h(s)s|\le g(s)s\le |h^{-1}(s)s|,\quad\text{ where }|s|\le 1,\]
where $h^{-1}$ is the inverse function of $h$.
\begin{theorem}[Energy decay rates \uppercase\expandafter{\romannumeral1}]\label{thedr}
In addition to all conditions of Theorem $\ref{thge2}$, $\mathbf{(H_5)}$ and $\mathbf{(H_6)}$, we assume $1\le m\le 5$. Then  for all $t\ge 0$
\begin{enumerate}[$(1)$]
  \item if $h(s)$ is linear,
    \begin{equation}
\label{a413}
E(t)\le E(0)e^{1-C\int_0^t\gamma(s)ds}.
\end{equation}
  \item if $h(s)$  has polynomial growth, 
  \begin{equation}
\label{b413}
E(t)\le  CE(0)\left(\frac{1}{1+\int_0^t\gamma(s)ds}\right)^{\frac{2}{m-1}} \text{with }m>1.
\end{equation}
\end{enumerate}
\end{theorem}

When $h(s)$ does not necessarily have polynomial growth, the energy decay rates still can be obtained if we replace $\mathbf{(H_5)}$ by the following

$\mathbf{(H_5')}$ $\gamma(t)\ge \gamma_0>0$, where $\gamma_0$ is constant.
\begin{theorem}[Energy decay rates \uppercase\expandafter{\romannumeral2}]\label{thedr1}
In addition to all conditions of Theorem $\ref{thge2}$, $\mathbf{(H_5')}$ and $\mathbf{(H_6)}$, we assume $1\le m\le 5$. Then for all $t\ge 1$
 \begin{equation}
\label{190}
E(t)\le CE(0)\left(H^{-1}\Big(\frac1t\Big)\right)^2.
\end{equation}
 Here $H(s):=h(s)s$.
 \end{theorem}

\begin{remark}
In the proof of Theorem \ref{thedr}, we only discuss the energy decay rates for $m\le 5$, that is, the 
nonlinear damping is \emph{subcritical} and \emph{critical}. The restraint results from the embedding theorem $H_0^1(\Omega) \hookrightarrow L^{m+1}(\Omega)$. A recent paper \cite{HT2022} investigated the energy decay estimates for the automous wave equation with \emph{supercritical} nonlinear
damping in the absence of the driving source. Inspired by the paper \cite{HT2022}, we reasonably conjecture that  if there exists $\kappa(t)$ satisfying some suitable conditions, then
\begin{equation*}
E(t)\le  CE(0)\left(\frac{1}{1+\int_0^t\kappa(t)\gamma(s)ds}\right)^{\frac{2}{m-1}},
\end{equation*}
for $m>5$, that is the supercritical nonlinear damping.
\end{remark}

\begin{remark}
In this paper, we discuss the  variable-coefficient wave equation with nonlinear time-dependent damping and nonlinear source for standard growth conditions. However, the energy decay rates for wave systems with nonstandard growth condition can be of equal importance. By following this paper, it is possible to discuss the energy decay rates to the initial-boundary value problem
 \begin{equation*}
\begin{cases}
u_{tt}-\Delta u+\gamma(t)|u_t|^{m(x)-2}u_t=|u|^{p(x)-2}u \quad &\text{ in }\Omega \times (0,T),\\
 u(x,t)=0\quad&\text{ on } \partial\Omega \times (0,T),\\
      u(x,0)=u_0(x),~u_t(x,0)=u_1(x)\quad&\text{ in }\Omega.
      \end{cases}
\end{equation*}
\end{remark}

\begin{remark}
The rest field $u(t,x)=0$ will be called \emph{asymptotically stable in the 
mean}, or simply \emph{asymptotically stable}, if and only if 
\[\lim_{t\to \infty} E(t)=0\text{ for all solutions } u(t):=u(t,x) \text{ of problem }\eqref{1.1}.\]
This concept was proposed first  by P. Pucci and J. Serrin \cite{2PS1996}. Obviously, based on Theorem \ref{thge1} or Theorem \ref{thge2}, by Lemma \ref{lem2}, we obtain $\lim_{t\to \infty} E(t)=0.$ Hence,  the rest field $u(x,t)=0$ is asymptotically stable.
\end{remark}

\section{Proof of Theorem \ref{thge2}}\label{sec3}

Let us introduce a function $\mathcal{F}$ as follows
\begin{equation}
\label{3.1}
\mathcal{F}(s)=\frac12s-\frac{M^{p+1}}{(p+1)\omega^{\frac{p+1}{2}}}s^{\frac{p+1}{2}}.
\end{equation}
By a direct computation, the function $\mathcal{F}$ satisfies  that
 \begin{enumerate}
  \item $\mathcal{F}(0)=0$;
  \item $\lim_{s\to +\infty}\mathcal{F}(s)=-\infty$;
  \item $\mathcal{F}$ is strictly increasing in $(0,s_1)$,  and is strictly decreasing in $(s_1,+\infty)$;
  \item $\mathcal{F}$ has a maximum at $s_1$ with the maximum value $\mathcal{F}_1$. Here 
  \[s_1=\Big(\frac{\omega^{\frac{p+1}{2}}}{M^{p+1}}\Big)^{\frac{2}{p-1}},\quad \mathcal{F}_1=\Big(\frac12-\frac{1}{p+1}\Big)s_1.\]
\end{enumerate}

\begin{lemma}\label{lem1}
If $u(t)$ is a solution for problem $(\ref{1.1})$ and
$E(0)<\mathcal{F}_1,$ $a(u_0,u_0)<s_1,$ then there exists a positive constant $s_2$ satisfying $0<s_2<s_1$ such that
\begin{equation}\label{2}
a(u(t),u(t))\le s_2\quad\text{ for all } t\ge0.
\end{equation}
\end{lemma}
\begin{proof}
Using \eqref{1.3} and \eqref{2.1}, the embedding theorem $H_0^1(\Omega) \hookrightarrow L^{p+1}(\Omega)$, we obtain 
\begin{equation}
\label{130}
\begin{split}
E(t)&\ge \frac12a(u(t),u(t))-\frac{1}{p+1}\|u\|_{p+1}^{p+1}\ge\frac12a(u(t),u(t))-\frac{M^{p+1}}{p+1}\|\nabla u\|_2^{p+1}\\
&\ge\frac12a(u(t),u(t))-\frac{M^{p+1}}{(p+1)\omega^{\frac{p+1}{2}}}[a(u(t),u(t))]^{\frac{p+1}{2}}:=\mathcal{F}(a(u(t),u(t))).
\end{split}
\end{equation}

Since $E(0)<\mathcal{F}_1$, there exists a $s_2<s_1$  such that $\mathcal{F}(s_2)=E(0)$. It follows from \eqref{130} that $\mathcal{F}(a(u_0,u_0))\leq E(0)=\mathcal{F}(s_2)$, which implies $a(u_0,u_0)\le s_2$ due to the given condition $a(u_0,u_0)<s_1$. To complete the proof of \eqref{2}, we suppose by contradiction that for some $t^0>0$, $a(u(t^0),u(t^0))>s_2.$ The continuity of $a(u(t),u(t))$ illustrates that we may choose $t^1$
such that $s_1>a(u(t^1),u(t^1))>s_2$, then we have
$E(0)=\mathcal{F}(s_2)<\mathcal{F}(a(u(t^1),u(t^1)))\leq E(t^1).$
This is a contradiction for $E(t)$ is nonincreasing.
\end{proof}

\begin{lemma}\label{lem2}
Under all the conditions of Lemma $\ref{lem1},$  for all $t\ge 0,$ the following holds
\begin{equation}
\label{5}
0\le \mathscr{E}(t)\le C_0E(t)\le C_0E(0).
\end{equation}
\end{lemma}
\begin{proof}
Similar to \eqref{130}, and then using \eqref{2} and $\eqref{2.4}$, then 
\begin{equation*}
\begin{split}
\frac{1}{p+1}\|u\|_{p+1}^{p+1}&\le  \frac{M^{p+1}}{p+1}\|\nabla u\|_2^{p+1}\le\frac{M^{p+1}}{(p+1)\omega^{\frac{p+1}{2}}}[a(u(t),u(t))]^{\frac{p+1}{2}}\\
&=\frac{M^{p+1}}{(p+1)\omega^{\frac{p+1}{2}}}[a(u(t),u(t))]^{\frac{p-1}{2}}a(u(t),u(t))\\
&\le \frac{M^{p+1}}{(p+1)\omega^{\frac{p+1}{2}}}s_2^{\frac{p-1}{2}}\Big(2E(t)+\frac{2}{p+1}\|u\|_{p+1}^{p+1}\Big)
,
\end{split}
\end{equation*}
which indicates 
\begin{equation}
\label{4}
\|u\|_{p+1}^{p+1}\le \mathcal{M}E(t)\le \mathcal{M}E(0),
\end{equation}
by recalling $E(t)$ is monotone decreasing. Here
\[\mathcal{M}:=\frac{ \frac{2M^{p+1}}{\omega^{\frac{p+1}{2}}}s_2^{\frac{p-1}{2}}}{1-\frac{2M^{p+1}}{(p+1)\omega^{\frac{p+1}{2}}}s_2^{\frac{p-1}{2}}}<\frac{ \frac{2M^{p+1}}{\omega^{\frac{p+1}{2}}}s_1^{\frac{p-1}{2}}}{1-\frac{2M^{p+1}}{(p+1)\omega^{\frac{p+1}{2}}}s_1^{\frac{p-1}{2}}}=\frac{2(p+1)}{p-1}.\]
Combining \eqref{4} with \eqref{2.4}, we directly obtain 
\[\mathscr{E}(t)=E(t)+\frac{1}{p+1}\|u\|^{p+1}\le C_0E(t)\le C_0E(0),\]
which yields \eqref{5}.
\end{proof}
 
 It follows from Lemma \ref{lem2} that the weak solution $u(t)$ of problem \eqref{1.1} exists globally, that is $T=\infty.$
 
 \begin{remark}
 By the well-known potential well theory, we can also prove the global existence. In fact, $\mathcal{F}_1$ is equal to the potential well depth(the mountain pass level) $d$ defined by 
 \[d:=\inf_{u\in H_0^1(\Omega)\backslash\{0\}}\sup_{\lambda\ge 0}J(\lambda u),\text{ where }J(u):=\frac12a(u,u)-\frac{1}{p+1}\|u\|_{p+1}^{p+1}.\]
 \end{remark}
 
\section{Proof of Theorem \ref{thedr}}\label{sec4}
The proof of Theorem \ref{thedr} relies on the following crucial lemma.  We refer to \cite{M1999} for the detailed proof.
 \begin{lemma}[\cite{M1999}]\label{lem4.1}
Let $E: \mathbb{R}^+\to\mathbb{R}^+$ be a non-increasing function and $\psi : \mathbb{R}^+\to\mathbb{R}^+$ be a strictly increasing function of class $C^1$ such that
\[\psi(0)=0\text{ and }\psi(t)\to+\infty\text{ as }t\to+\infty.\]
Assume that there exist $\sigma \geq  0$, $\sigma' \geq  0$, $c \geq  0$ and $\omega > 0$ such that$:$
 \[\int_t^{+\infty}E^{1+\sigma}(s)\psi^{\prime}(s)ds\leq  \frac{1}{\omega}E^\sigma(0) E(t)+\frac{c}{(1+\psi(s))^{\sigma'}}E^q(0)E(s)\]
 then $E$ has the following decay property$:$
\begin{enumerate}[$(1)$]
  \item if $\sigma=c=0,$ then $E(t)\leq  E(0)e^{1-\omega\psi(t)}$ for all $t\geq  0;$
  \item if $\sigma>0,$ then there exists $C>0$ such that $E(t)\leq  CE(0)(1+\psi(t))^{-\frac{1+\sigma'}{\sigma}}$ for all $t\geq  0$.
\end{enumerate}
\end{lemma}

In what follows, let us prove Theorem \ref{thedr}. It dose make sense to multiply (\ref{1.1}) by $E^\beta(t)\phi'(t)u(t)$  and to integrate over $\Omega\times[S,T]$. Here $\phi:\mathbb{R}^+\to \mathbb{R}^+$ is
a concave nondecreasing function of class $C^2$, such that $\phi'$ is bounded, and $\beta\ge 0$ is constant. Then we obtain 
\begin{equation}\label{4.1}
\begin{split}
\int_S^T E^\beta(t)\phi'(t)\int_\Omega u(t)[u_{tt}(t)-\mathcal{A}u(t)+\gamma(t)g(u_t(t))]dxdt=\int_S^T E^\beta(t)\phi'(t)\|u(t)\|_{p+1}^{p+1}dt.
\end{split}\end{equation}
Integrating by parts yields
\begin{equation}
\label{4.2}
\begin{split}
&\int_S^T E^\beta(t)\phi'(t)\int_\Omega u(t)u_{tt}(t)dxdt=E^\beta(t)\phi'(t)\int_\Omega u(t)u_{t}(t)dx\Big|_S^T-\int_S^T E^\beta(t)\phi'(t)\|u_t(t)\|_2^2\\
&\quad-\int_S^T[\beta E^{\beta-1}(t)E'(t)\phi'(t)+E^{\beta}(t)\phi''(t)]\int_\Omega u(t)u_{t}(t)dxdt.
\end{split}
\end{equation}
Substituting \eqref{4.2} into \eqref{4.1}, one obtains
\begin{equation}
\label{4.3}
\begin{split}
&E^\beta(t)\phi'(t)\int_\Omega u(t)u_{t}(t)dx\Big|_S^T-\int_S^T E^\beta(t)\phi'(t)\|u_t(t)\|_2^2dt\\
&\quad-\int_S^T[\beta E^{\beta-1}(t)E'(t)\phi'(t)+E^{\beta}(t)\phi''(t)]\int_\Omega u(t)u_{t}(t)dxdt\\
&\quad+\int_S^T E^\beta(t)\phi'(t)a(u(t),u(t))dt\\
&\quad+\int_S^T E^\beta(t)\phi'(t)\gamma(t)\int_\Omega u(t)g(u_t(t))dxdt=\int_S^T E^\beta(t)\phi'(t)\|u(t)\|_{p+1}^{p+1}dt.
\end{split}
\end{equation}
It follows from \eqref{2.4} that \eqref{4.3} can be written as
\begin{equation}
\label{4.4}
\begin{split}
&2\int_S^T E^{\beta+1}(t)\phi'(t)dt=-E^\beta(t)\phi'(t)\int_\Omega u(t)u_{t}(t)dx\Big|_S^T+2\int_S^T E^\beta(t)\phi'(t)\|u_t(t)\|_2^2dt\\
&\quad+\int_S^T[\beta E^{\beta-1}(t)E'(t)\phi'(t)+E^{\beta}(t)\phi''(t)]\int_\Omega u(t)u_{t}(t)dxdt\\
&\quad+\frac{p-1}{p+1}\int_S^T E^\beta(t)\phi'(t)\|u(t)\|_{p+1}^{p+1}dt-\int_S^T E^\beta(t)\phi'(t)\gamma(t)\int_\Omega u(t)g(u_t(t))dxdt\\
&=\mathcal{J}_1+\mathcal{J}_2+\mathcal{J}_3+\mathcal{J}_4+\mathcal{J}_5.
\end{split}
\end{equation}

In what follows, let us estimate every term on the right hand side above.
 
Using the embedding theorem $H_0^1(\Omega) \hookrightarrow L^{2}(\Omega)$, \eqref{1.3} and \eqref{5} yields
\begin{equation}
\label{4.5}
\begin{split}
&\|u(t)\|_2\le M\|\nabla u(t)\|_2\le M\Big[\frac{a(u(t),u(t))}{\omega}\Big]^{\frac12}\le C_1E^{\frac{1}{2}}(t).
\end{split}
\end{equation}

Applying Cauchy's inequality, the boundedness of $\phi'(t)$(we can denote by $\theta$), \eqref{4.5} and \eqref{5}, we arrive at
\begin{equation}
\label{4.6}
\begin{split}
|\mathcal{J}_1|&=\Big|-E^\beta(t)\phi'(t)\int_\Omega u(t)u_{t}(t)dx\Big|_S^T\Big|\\
&\le \theta E^\beta(S)\big[\|u_t(T)\|_2\|u(T)\|_2+\|u_t(S)\|_2\|u(S)\|_2\big]\le C_2E^{\beta+1}(S),
\end{split}
\end{equation}
moreover,
\begin{equation}
\label{4.7}
\begin{split}
|\mathcal{J}_3|&=\Big|\int_S^T[\beta E^{\beta-1}(t)E'(t)\phi'(t)+E^{\beta}(t)\phi''(t)]\int_\Omega u(t)u_{t}(t)dxdt\Big|\\
&\le -\theta\beta\int_S^T E^{\beta}(t) E'(t)dt+C_3\int_S^TE^{\beta+1}(t)(-\phi''(t))dt\\
&\le \frac{\theta\beta }{\beta+1}E^{\beta+1}(S) +C_4E^{\beta+1}(S)=C_5E^{\beta+1}(S).
\end{split}
\end{equation}

 It follows from \eqref{4} that
\begin{equation}
\label{4.8}
\begin{split}
|\mathcal{J}_4|=\frac{p-1}{p+1}\int_S^T E^\beta(t)\phi'(t)\|u(t)\|_{p+1}^{p+1}dt\le \frac{p-1}{p+1}\mathcal{M}\int_S^T E^{\beta+1}(t)\phi'(t)dt.
\end{split}
\end{equation}

Using Young's inequality with $\varepsilon_1>0$, \eqref{4.5}, one has
\begin{equation*}
\begin{split}
\int_{|u_t(t)|\le1} u(t)g(u_t(t))dx&\le \frac{\varepsilon_1}{2}\|u(t)\|_2^2+\frac{1}{2\varepsilon_1}\int_{|u_t(t)|\le1} |g(u_t(t))|^2dx\\
&\le C_6\varepsilon_1E(t)+\frac{1}{2\varepsilon_1}\int_{|u_t(t)|\le1} |g(u_t(t))|^2dx.
\end{split}
\end{equation*}
It follows from Young's inequality with $\varepsilon_1>0$, the embedding theorem $H_0^1(\Omega) \hookrightarrow L^{m+1}(\Omega)$, \eqref{1.3} and \eqref{5} that
\begin{equation*}
\begin{split}
&\int_{|u_t(t)|>1} u(t)g(u_t(t))dx\le \frac{\varepsilon_2}{m+1}\|u(t)\|_{m+1}^{m+1}+\frac{m\varepsilon_2^{-\frac1m}}{m+1}\int_{|u_t(t)|>1} |g(u_t(t))|^{\frac{m+1}{m}}dx\\
& \le \frac{M^{m+1}\varepsilon_2}{m+1}\|\nabla u(t)\|_2^{m+1}+\frac{m\varepsilon_2^{-\frac1m}}{m+1}\int_{|u_t(t)|>1} |g(u_t(t))|^{\frac{m+1}{m}}dx\\
& \le \frac{M^{m+1}\varepsilon_2}{m+1}\Big[\frac{a(u(t),u(t))}{\omega}\Big]^{\frac{m+1}{2}}+\frac{m\varepsilon_2^{-\frac1m}}{m+1}\int_{|u_t(t)|>1} |g(u_t(t))|^{\frac{m+1}{m}}dx
\\
& \le C_7\varepsilon_2E(t)+\frac{m\varepsilon_2^{-\frac1m}}{m+1}\int_{|u_t(t)|>1} |g(u_t(t))|^{\frac{m+1}{m}}dx.
\end{split}
\end{equation*}
Therefore, by recalling the boundedness of $\gamma(t)$(we can denote by $c_1$), then
\begin{equation}
\label{4.9}
\begin{split}
|\mathcal{J}_5|=&\Big|-\int_S^T E^\beta(t)\phi'(t)\gamma(t)\int_\Omega u(t)g(u_t(t))dxdt\Big|\\
&\le c_1\int_S^T E^\beta(t)\phi'(t)\Big|\int_{|u_t(t)|\le 1} u(t)g(u_t(t))dx+\int_{|u_t(t)|>1} u(t)g(u_t(t))dx\Big|dt\\
&\le c_1\int_S^T E^\beta(t)\phi'(t)\Big[C_6\varepsilon_1E(t)+\frac{1}{2\varepsilon_1}\int_{|u_t(t)|\le1} |g(u_t(t))|^2dx\\
&\quad+C_7\varepsilon_2E(t)+\frac{m\varepsilon_2^{-\frac1m}}{m+1}\int_{|u_t(t)|>1} |g(u_t(t))|^{\frac{m+1}{m}}dx\Big]dt\\
&\le C_8(\varepsilon_1+\varepsilon_2)\int_S^T E^{\beta+1}\phi'(t)dt+\frac{c_1}{2\varepsilon_1}\int_S^T E^\beta(t)\phi'(t)\int_{|u_t(t)|\le1} |g(u_t(t))|^2dx\\
&\quad+\frac{c_1m\varepsilon_2^{-\frac1m}}{m+1}\int_S^T E^\beta(t)\phi'(t)\int_{|u_t(t)|>1} |g(u_t(t))|^{\frac{m+1}{m}}dxdt.
\end{split}
\end{equation}

Inserting \eqref{4.6}-\eqref{4.9} into \eqref{4.4} indicates
 \begin{equation}
\label{4.10}
\begin{split}
&\Big[2-\frac{p-1}{p+1}\mathcal{M}-C_8(\varepsilon_1+\varepsilon_2)\Big]\int_S^T E^{\beta+1}(t)\phi'(t)dt\\
&\le C_9E^{\beta+1}(S)+2\int_S^T E^\beta(t)\phi'(t)\|u_t(t)\|_2^2dt\\
&\quad+\frac{c_1}{2\varepsilon_1}\int_S^T E^\beta(t)\phi'(t)\int_{|u_t(t)|\le1} |g(u_t(t))|^2dxdt\\
&\quad+\frac{c_1m\varepsilon_2^{-\frac1m}}{m+1}\int_S^T E^\beta(t)\phi'(t)\int_{|u_t(t)|>1} |g(u_t(t))|^{\frac{m+1}{m}}dxdt.
\end{split}
\end{equation}
Note that $2-\frac{p-1}{p+1}\mathcal{M}>0$,  one gets 
\[2-\frac{p-1}{p+1}\mathcal{M}-C_8(\varepsilon_1+\varepsilon_2)>0\]
for sufficiently small positive constants $\varepsilon_1$ and $\varepsilon_2$.

\subsection{Case \uppercase\expandafter{\romannumeral1}: $h(s)$ is linear}

Let us  $\phi(t)=\int_0^t\gamma(s)ds$ in this section.  From $\mathbf{(H_6)}$, there exists two positive constant $b_3,~b_4$ such that 
\begin{equation}
\label{s4.1}
b_3s^2\le g(s)s\le b_4s^2,\quad\text{ where }|s|\le 1.
\end{equation}
Using \eqref{s4.1}, $\mathbf{(H_2)}$ and \eqref{4111} yields
\begin{equation}
\label{s4.2}
\begin{split}
&2\int_S^T E^\beta(t)\gamma(t)\|u_t(t)\|_2^2dt\\
&\le 2\int_S^T E^\beta(t)\gamma(t)\Big[\int_{|u_t(t)|\le 1}|u_t(t)|^2dx+
\int_{|u_t(t)|>1}|u_t(t)|^2dx\Big]dt\\
&\le \frac{2}{b_3}\int_S^T E^\beta(t)\int_\Omega\gamma(t)g(u_t(t))u_t(t)dxdt+2\int_S^T E^\beta(t)\gamma(t)\int_{|u_t(t)|>1} |u_t(t)|^{m+1}dxdt\\
&\le \frac{2}{b_3}\int_S^T E^\beta(t)\int_\Omega\gamma(t)g(u_t(t))u_t(t)dxdt+\frac{2}{b_1}\int_S^T E^\beta(t)\gamma(t)\int_\Omega g(u_t(t))u_t(t)dxdt\\
&\le -\Big[\frac{2}{b_3}+\frac{2}{b_1}\Big]\int_S^T E^\beta(t)E'(t)dt\le C_{10}E^{\beta+1}(S).
\end{split}
\end{equation}
By using \eqref{s4.1}  and \eqref{4111}, we gets
\begin{equation}
\label{s4.3}
\begin{split}
&\frac{c_1}{2\varepsilon_1}\int_S^T E^\beta(t)\gamma(t)\int_{|u_t(t)|\le1} |g(u_t(t))|^2dx
\\&\le \frac{c_1b_4}{2\varepsilon_1}\int_S^T E^\beta(t)\gamma(t)\int_{|u_t(t)|\le1} g(u_t(t))u_t(t)dx\\
&\le -\frac{c_1b_4}{2\varepsilon_1}\int_S^T E^\beta(t)E'(t)dt\le C_{11}E^{\beta+1}(S).
\end{split}
\end{equation}
It follows from $\mathbf{(H_2)}$ and \eqref{4111} that
\begin{equation}
\label{s4.4}
\begin{split}
&\frac{c_1m\varepsilon_2^{-\frac1m}}{m+1}\int_S^T E^\beta(t)\gamma(t)\int_{|u_t(t)|>1} |g(u_t(t))|^{\frac{m+1}{m}}dxdt\\
&=\frac{c_1m\varepsilon_2^{-\frac1m}}{m+1}\int_S^T E^\beta(t)\gamma(t)\int_{|u_t(t)|>1} |g(u_t(t))|^{\frac{1}{m}}|g(u_t(t))|dxdt
\\
&\le\frac{c_1m\varepsilon_2^{-\frac1m}b_2^{\frac1m}}{(m+1)}\int_S^T E^\beta(t)\gamma(t)\int_{|u_t(t)|>1} g(u_t(t))u_t(t)dxdt\le C_{12}E^{\beta+1}(S).
\end{split}
\end{equation}

Inserting \eqref{s4.2}-\eqref{s4.4} into \eqref{4.10}, we easily deduce
\begin{equation}
\label{s4.5}
\int_S^T E^{\beta+1}(t)\gamma(t)dt\le \frac{1}{C_{13}}E^\beta(0)E(S).
\end{equation}
When $T$ goes to $+\infty$, choose $\beta=0$, $\psi(t)=\int_0^t\gamma(s)ds$ in Lemma \ref{lem4.1}, then we derive \eqref{a413}.

\subsection{Case \uppercase\expandafter{\romannumeral2}: $h(s)$  has polynomial growth}

In this subsection, we still choose  $\phi(t)=\int_0^t\gamma(s)ds$. To simplicity, denote $h(s)=b_5|s|^{m-1}s$ with $m>1,~ |s|\le 1,~b_5>0$. From $\mathbf{(H_6)}$, there exists two positive constant $b_6,~b_7$ such that 
\begin{equation}
\label{s4.6}
b_6s^{m+1}\le g(s)s\le b_7|s|^{\frac{m+1}{m}},\quad\text{ where }|s|\le 1.
\end{equation}
It follows from H\"older's inequality, \eqref{s4.6}, Young's inequality with $\varepsilon_3>0$, $\mathbf{(H_2)}$ and \eqref{4111} yields
\begin{equation}
\label{s4.7}
\begin{split}
&2\int_S^T E^\beta(t)\gamma(t)\|u_t(t)\|_2^2dt\\
&\le 2\int_S^T E^\beta(t)\gamma(t)\Big[\int_{|u_t(t)|\le 1}|u_t(t)|^2dx+
\int_{|u_t(t)|>1}|u_t(t)|^2dx\Big]dt\\
&\le 2|\Omega|^{\frac{m-1}{m+1}}\int_S^T E^\beta(t)\gamma(t)\Big(\int_{|u_t(t)|\le 1}|u_t(t)|^{m+1}dx\Big)^{\frac{2}{m+1}}dt+\frac{2}{b_1}\int_S^T E^\beta(t)\gamma(t)\int_\Omega g(u_t(t))u_t(t)dxdt
\\
&\le 2|\Omega|^{\frac{m-1}{m+1}}\frac{1}{b_6^{\frac{2}{m+1}}}\int_S^T E^\beta(t)\gamma^{\frac{m-1}{m+1}}(t)\Big(\int_\Omega \gamma(t)u_t(t)g(u_t(t))dx\Big)^{\frac{2}{m+1}}dt-\frac{2}{b_1}\int_S^T E^\beta(t)E'(t)dxdt\\
&\le C_{14}\varepsilon_3\int_S^T E^{\frac{\beta(m+1)}{m-1}}(t)\gamma(t)dt+ C_{15}\varepsilon_3^{-\frac{m-1}{2}}\int_S^T (-E'(t))dt+C_{16}E^{\beta+1}(S)\\
&\le C_{14}\varepsilon_3\int_S^T E^{\frac{\beta(m+1)}{m-1}}(t)\gamma(t)dt+ C_{17}E(S)+C_{16}E^{\beta+1}(S).
\end{split}
\end{equation}
Similar to \eqref{s4.7},  using the second inequality in \eqref{s4.6}, we obtain 
\begin{equation}
\label{s4.8}
\begin{split}
&\frac{c_1}{2\varepsilon_1}\int_S^T E^\beta(t)\gamma(t)\int_{|u_t(t)|\le1} |g(u_t(t))|^2dx\\
&\le \frac{c_1}{2\varepsilon_1}|\Omega|^{\frac{m-1}{m+1}}\int_S^T E^\beta(t)\gamma(t)\Big(\int_{|u_t(t)|\le 1}|g(u_t(t))|^{m+1}dx\Big)^{\frac{2}{m+1}}dt
\\
&\le \int_S^T E^\beta(t)\gamma^{\frac{m-1}{m+1}}(t)\cdot\frac{c_1}{2\varepsilon_1}|\Omega|^{\frac{m-1}{m+1}}b_7^\frac{2m}{m+1}\Big(\int_\Omega \gamma(t)u_t(t)g(u_t(t))dx\Big)^{\frac{2}{m+1}}dt\\
&\le C_{18}\varepsilon_4\int_S^T E^{\frac{\beta(m+1)}{m-1}}(t)\gamma(t)dt+ C_{19}\varepsilon_4^{-\frac{m-1}{2}}\int_S^T (-E'(t))dt\\
&\le C_{18}\varepsilon_4\int_S^T E^{\frac{\beta(m+1)}{m-1}}(t)\gamma(t)dt+ C_{20}E(S).
\end{split}
\end{equation}
Moreover, \eqref{s4.4} still holds in this case.

Inserting \eqref{s4.4}, \eqref{s4.7} and \eqref{s4.8} into \eqref{4.10}, and choosing $\beta=\frac{m-1}{2}$, we easily deduce
 \begin{equation}
\label{s4.9}
\begin{split}
\Big[2-\frac{p-1}{p+1}\mathcal{M}-C_8(\varepsilon_1+\varepsilon_2)-C_{21}(\varepsilon_3+\varepsilon_4)\Big]\int_S^T E^{\beta+1}(t)\gamma(t)dt\le C_{22}E^{\beta+1}(0)E(S).
\end{split}
\end{equation}
We also have 
\[2-\frac{p-1}{p+1}\mathcal{M}-C_8(\varepsilon_1+\varepsilon_2)-C_{21}(\varepsilon_3+\varepsilon_4)>0\]
for sufficiently small positive constants $\varepsilon_1$, $\varepsilon_2$, $\varepsilon_3$ and $\varepsilon_4$.

Therefore, we get
\begin{equation}
\label{s4.5}
\int_S^T E^{\beta+1}(t)\gamma(t)dt\le \frac{1}{C_{23}}E^\beta(0)E(S).
\end{equation}
When $T$ goes to $+\infty$, note that $\beta=\frac{m-1}{2}$,  $\psi(t)=\int_0^t\gamma(s)ds$ in Lemma \ref{lem4.1} then we derive    \eqref{b413}.

\section{Proof of Theorem \ref{thedr1}}\label{sec5}

To prove Theorem \ref{thedr1}, we need to the following lemma included in \cite{M1999}.
\begin{lemma}\label{lem4.2}
The function $\phi:[1,+\infty)\to [1,\infty)$ defined by 
\begin{equation}
\label{193}
\phi(t):=\tilde{\psi}^{-1}(t)
\end{equation}
is a strictly increasing concave function of class $C^2$ and satisfies
\[\lim_{t\to+\infty}\phi(t)=+\infty,~\lim_{t\to+\infty}\phi'(t)=0,\]
and $\int_1^\infty \phi'(t) |h^{-1}(\phi'(t))|^2dt$ converges.
Here \begin{equation}
\label{194}
\tilde{\psi}(t):=1+\int_1^t\frac{1}{h(1/s)}ds, 
\end{equation}and the function $h$ satisfies $\mathbf{(H_6)}$.
\end{lemma}

In this section, define 
\begin{equation}
\label{191}
\chi(t):=\frac{1}{\phi(t)}=\frac{1}{\tilde{\psi}^{-1}(t)}.
\end{equation}

Now we are in a position to estimate \eqref{4.10}.  \\

For $t\ge 1$, set
\begin{equation*}
\begin{split}
\Omega^1&=\{x\in \Omega:~|u_t(t)|\le \chi(t)\};\\
\Omega^2&=\{x\in \Omega:~\chi(t)<|u_t(t)|\le 1\};\\
\Omega^3&=\{x\in \Omega:~|u_t(t)|>1\}.
\end{split}
\end{equation*}
Using Lemma \ref{lem4.2}, $\mathbf{(H_6)}$, $\mathbf{(H_5')}$ and \eqref{4111}, we gets
\begin{equation}
\label{140}
\begin{split}
&\int_S^T E^\beta(t)\int_{\Omega^2}\phi'(t)|u_t(t)|^2dxdt\\
&= \int_S^T E^\beta(t)\int_{\Omega^2}h(\chi(t))|u_t(t)|^2dxdt
\le \int_S^T E^\beta(t)\int_{\Omega^2}h(u_t(t))|u_t(t)|^2dxdt\\
&\le \int_S^T E^\beta(t)\frac{1}{\gamma(t)}\int_{\Omega^2}\gamma(t)g(u_t(t))u_t(t)dxdt\le
\frac{1}{\gamma_0}E^{\beta+1}(S).
\end{split}
\end{equation}
It follows from \eqref{4111}, \eqref{191} and Lemma \ref{lem4.2} that
\begin{equation}
\label{141}
\begin{split}
&\int_S^T E^\beta(t)\phi'(t)\int_{\Omega^1}|u_t(t)|^2dxdt\le 
|\Omega|\int_S^T E^\beta(t)\phi'(t)\chi^2(t)dt\\
&\le 
|\Omega|E^\beta(S)\int_S^T \phi'(t)\chi^2(t)dt\le |\Omega|\frac{E^\beta(S)}{\phi(S)}.
\end{split}
\end{equation}
Recalling \eqref{s4.7}, and using \eqref{140} and \eqref{141}, one gets
\begin{equation}
\label{142}
\begin{split}
&2\int_S^T E^\beta(t)\phi'(t)\|u_t(t)\|_2^2dt\\
&\le 2\int_S^T E^\beta(t)\phi'(t)\int_{|u_t(t)|\le 1}|u_t(t)|^2dxdt+\frac{2}{b_1}\int_S^T E^\beta(t)\frac{\phi'(t)}{\gamma(t)}\int_{\Omega^3} \gamma(t)g(u_t(t))u_t(t)dxdt\\
&\le 2\int_S^T E^\beta(t)\phi'(t)\int_{\Omega^1\cup \Omega^2}|u_t(t)|^2dxdt+C_{24}E^{\beta+1}(S)\\
&\le C_{25}E^{\beta+1}(S)+ |\Omega|\frac{E^\beta(S)}{\phi(S)}.
\end{split}
\end{equation}

Next we continue to estimate the remaining terms in \eqref{4.10}. \\
For $t\ge 1$ and $\phi'(t)\le 1$, set
\begin{equation*}
\begin{split}
\Omega_1&=\{x\in \Omega:~|u_t(t)|\le \phi'(t)\};\\
\Omega_2&=\{x\in \Omega:~\phi'(t)<|u_t(t)|\le 1\};\\
\Omega_3&=\{x\in \Omega:~|u_t(t)|>1\}.
\end{split}
\end{equation*}
For $t\ge 1$ and $\phi'(t)> 1$, set
\begin{equation*}
\begin{split}
\Omega_4&=\{x\in \Omega:~|u_t(t)|\le 1<\phi'(t)\};\\
\Omega_5&=\{x\in \Omega:~1<|u_t(t)|\le \phi'(t)\};\\
\Omega_6&=\{x\in \Omega:~|u_t(t)|> \phi'(t)>1\}.
\end{split}
\end{equation*}
Hence 
\[\{x\in \Omega:~|u_t(t)|\le 1\}=\Omega_1\cup\Omega_2(\text{or }\Omega_4),~\{x\in \Omega:~|u_t(t)|>1\}=\Omega_3(\text{or }\Omega_5\cup\Omega_6).\]
Similar to \eqref{s4.4},
\begin{equation*}
\begin{split}
\frac{c_1m\varepsilon_2^{-\frac1m}}{m+1}\int_S^T E^\beta(t)\phi'(t)\int_{\Omega^i} |g(u_t(t))|^{\frac{m+1}{m}}dxdt\le C_{26}E^{\beta+1}(S), \quad \text{i=3,5,6}.
\end{split}
\end{equation*}
Thus
\begin{equation}
\label{143}
\begin{split}
\frac{c_1m\varepsilon_2^{-\frac1m}}{m+1}\int_S^T E^\beta(t)\phi'(t)\int_{|u_t(t)|>1} |g(u_t(t))|^{\frac{m+1}{m}}dxdt\le C_{27}E^{\beta+1}(S).
\end{split}
\end{equation}
Using $\mathbf{(H_6)}$ and $\mathbf{(H_5')}$, then
\begin{equation}
\label{144}
\begin{split}
&\frac{c_1}{2\varepsilon_1}\int_S^T E^\beta(t)\phi'(t)\int_{\Omega_2} |g(u_t(t))|^2dxdt\\
&\le\frac{c_1}{2\varepsilon_1}\int_S^T E^\beta(t)\int_{\Omega_2} |u_t(t)| |g(u_t(t))|^2dxdt\\
&\le\frac{c_1}{2\varepsilon_1}\int_S^T E^\beta(t)\int_{\Omega_2} u_t(t)g(u_t(t))|h^{-1}(u_t(t))|dxdt
\\
&\le\frac{c_1}{2\varepsilon_1}h^{-1}(1)\int_S^T E^\beta(t)\frac{1}{\gamma(t)}\int_{\Omega_2} \gamma(t)u_t(t)g(u_t(t))dxdt\le C_{28}E^{\beta+1}(S),
\end{split}
\end{equation}
and 
\begin{equation}
\label{145}
\begin{split}
&\frac{c_1}{2\varepsilon_1}\int_S^T E^\beta(t)\phi'(t)\int_{\Omega_i} |g(u_t(t))|^2dxdt\\
&\le\frac{c_1}{2\varepsilon_1}\int_S^T E^\beta(t)\phi'(t)\int_{\Omega_i} |h^{-1}(u_t(t))|^2dxdt\\
&\le\frac{c_1}{2\varepsilon_1}\int_S^T E^\beta(t)\phi'(t)\int_{\Omega_i} |h^{-1}(\phi'(t))|^2dxdt\\
&\le C_{29}E^{\beta}(S)\int_S^T \phi'(t) |h^{-1}(\phi'(t))|^2dt,  \quad \text{i=1,4}.
\end{split}
\end{equation}

Substubiting \eqref{142}-\eqref{145} into \eqref{4.10}, we get 
\begin{equation}
\label{146}
\begin{split}
&\Big[2-\frac{p-1}{p+1}\mathcal{M}-C_8(\varepsilon_1+\varepsilon_2)\Big]\int_S^T E^{\beta+1}(t)\phi'(t)dt\\&\le C_{30}E^{\beta+1}(S)+ |\Omega|\frac{E^\beta(S)}{\phi(S)}+C_{31}E^{\beta}(S)\int_S^T \phi'(t) |h^{-1}(\phi'(t))|^2dt.
\end{split}
\end{equation}
Since $\int_1^\infty \phi'(t) |h^{-1}(\phi'(t))|^2dt$ converges in Lemma \ref{lem4.2},  then \begin{equation}
\label{147}
\int_S^T E^{\beta+1}(t)\phi'(t)dt\le \frac{1}{C_{32}}E^\beta(0)E(S)+\frac{C_{33}}{\phi(S)}E^\beta(0)E(S).
\end{equation}
When $T$ goes to $+\infty$, choose $\beta=1$, and choose  $\psi(t)=\phi(t)-1$ in Lemma \ref{lem4.1}, then we derive 
\begin{equation}
\label{192}
E(t)\le \frac{CE(0)}{\phi^2(t)}.
\end{equation}
Let us choose $s_0$ such that $h(\frac{1}{s_0})\le 1$, then by \eqref{194} and $H(s):=h(s)s$, for $s\ge s_0$
\[\tilde{\psi}(s)\le 1+(s-1)\frac{1}{h\Big(\frac{1}{s}\Big)}\le s\frac{1}{h\Big(\frac{1}{s}\Big)}=\frac{1}{H\Big(\frac{1}{s}\Big)}.\]
Hence, by \eqref{193}, for $s\ge s_0$
\[s\le \phi\left(\frac{1}{H\Big(\frac{1}{s}\Big)}\right)=\phi(t)\quad \text{ with }t=\frac{1}{H\Big(\frac{1}{s}\Big)}.\]
Further, 
\[\frac{1}{\phi(t)}\le \frac1s=H^{-1}\Big(\frac1t\Big).\]
which together with \eqref{192} implies  \eqref{190}.

\section*{Acknowledgements} This work is supported by the Fundamental Research Funds for Central Universities(B230201033). 

\section*{Competing Interests}
The authors declare that they have no competing interests.

\section*{Data Availability}
Data sharing is not applicable to this article as no new data were created or analyzed in this study.


\begin{thebibliography}{00}

\bibitem{D2011}
M. Daoulatli, \emph{Rates of decay for the wave systems with time-dependent damping}, Discrete Contin. Dyn. Syst., \textbf{31} (2011) 407--443.

\bibitem{GR2014}
Y.Q. Guo, M.A. Rammaha, et al., \emph{Hadamard well-posedness for a hyperbolic equation of viscoelasticity with supercritical sources and damping}, J. Differential Equations, \textbf{257} (2014), 3778--3812.

\bibitem{H2021}
T.G. Ha, \emph{Global solutions and blow-up for the wave equation with variable
coefficients: \uppercase\expandafter{\romannumeral1}. Interior supercritical source.} Appl. Math. Optim., \textbf{84} (2021), 
767--803.

\bibitem{HT2022}
A. Haraux, L. Tebou, {\em Energy decay estimates for the wave equation with supercritical nonlinear damping}, arXiv preprint arXiv:2204.11494.




\bibitem{LX2020}
J.R. Luo, T.J. Xiao, \emph{Decay rates for second order evolution equations in Hilbert spaces with nonlinear time-dependent
damping}, Evol. Equ. Control Theory, \textbf{9} (2020), 359--373.

\bibitem{LX2021}
J.R. Luo, T.J. Xiao, \emph{Decay rates for semilinear wave equations with vanishing damping and Neumann boundary
conditions}, Math. Methods Appl. Sci., \textbf{44} (2021), 303--314.

\bibitem{2LX2021}
J.R. Luo, T.J. Xiao, \emph{Optimal energy decay rates for abstract second
order evolution equations with non-autonomous
damping}, ESAIM: COCV, \textbf{27} (2021) 59.

\bibitem{LX2023}
J.R. Luo, T.J. Xiao, 
\emph{Optimal decay rates for semi-linear non-autonomous evolution
equations with vanishing damping}, Nonlinear Anal., \textbf{230} (2023) 113247.

\bibitem{M1999}
P. Martinez, {\em A new method to obtain decay rate estimates for dissipative systems}, ESAIM: Control Optim. Calc. Var., \textbf{4} (1999), 419--444.

\bibitem{M2000}
P. Martinez, \emph{Precise decay rate estimates for time-dependent dissipative systems}, Israel J. Math., \textbf{119} (2000), 291--324.

\bibitem{PS1996}
P. Pucci and J. Serrin, \emph{Asymptotic stability for non–autonomous dissipative wave systems}, Comm. Pure Appl. Math., \textbf{XLIX} (1996), 177--216.

\bibitem{2PS1996}
P. Pucci and J. Serrin, \emph{Asympptotic stablility for nonlinear parabolic systems}//S.N. Antontsev, J.I. Diaz,
S.I. Shmarev. Energy Methods in Continuum Mechanics. Dordrecht: Kluwer Acad Publ, 1996: 66--74.

\bibitem{PS1998}
P. Pucci and J. Serrin, \emph{Local asymptotic stability for dissipative wave systems}, Israel J. Math., \textbf{104} (1998), 29--50.

\end{thebibliography}
\end{document}